\documentclass{amsart}

\newtheorem{theorem}{Theorem}[section]
\newtheorem{lemma}[theorem]{Lemma}
\newtheorem{proposition}[theorem]{Proposition}

\theoremstyle{definition}
\newtheorem{definition}[theorem]{Definition}
\newtheorem{notation}[theorem]{Notation}

\theoremstyle{remark}

\numberwithin{equation}{section}

\usepackage{color}

\usepackage{amscd,amsthm,amsmath,amssymb,amsfonts}

\usepackage{mathrsfs}

\def\hv{{\widehat v}}

\def\eps{{\varepsilon}}

\def\bff{{\bf f}}

\def\rme{{\rm e}}

\def\bfx{{\bf x}}

\def\bff{{\bf f}}
\def\bfs{{\bf s}}

\def\beq{\begin{equation}}
\def\eeq{\end{equation}}

\def\bfa{{\bf a}}

\def\rep{{\rm rep}}
\def\Rep{{\rm Rep}}
\def\rme{{\rm e}}

\def\sm{{\smallskip}}

\def\beq{\begin{equation}}
\def\eeq{\end{equation}}

\begin{document}

\title[On the $b$-ary expansion of a real number]
{On the $b$-ary expansion of a real number whose irrationality exponent is close to $2$}

\author{Yann Bugeaud}
\address{I.R.M.A., UMR 7501, Universit\'e de Strasbourg
et CNRS, 7 rue Ren\'e Descartes, 67084 Strasbourg Cedex, France}
\address{Institut universitaire de France}
\email{bugeaud@math.unistra.fr}

\author{Dong Han Kim}
\address{Department of Mathematics Education,
Dongguk University, Seoul 04620, Korea.}
\email{kim2010@dgu.ac.kr}

\begin{abstract}
Let $b \ge 2$ be an integer and $\xi$ an irrational real number. 
We establishes that, if the irrationality exponent of $\xi$ is less than $2.324 \ldots$,   
then the $b$-ary expansion of $\xi$ cannot be `too simple', in a suitable sense. 
This improves the results of our previous paper [Ann. Sc. Norm. Super. Pisa Cl. Sci., 2017]. 
\end{abstract}


\maketitle

\section{Introduction}\label{sec:1}

A central question in Diophantine approximation is to determine how well a given 
irrational real number $\xi$ can be approximated by rational numbers. 

\begin{definition}\label{defirrexp}
The irrationality exponent $\mu(\xi)$ of an irrational real number $\xi$ is the supremum
of the real numbers $\mu$ such that the inequality
$$
\biggl| \xi - \frac{p}{q} \biggr| < \frac{1}{q^{\mu}}
$$
has infinitely many solutions in rational numbers $p/q$.
\end{definition} 

It follows from the theory of continued fractions that every 
convergent $p/q$ of $\xi$ satisfies $|\xi - p/q| < 1/q^2$. Consequently, we get the lower bound $\mu (\xi) \ge 2$. 
In the opposite direction, 
an easy covering argument, usually called the Cantelli lemma, shows that we have 
$\mu (\xi) \le 2$ (and thus $\mu (\xi) = 2$) for 
almost all $\xi$, with respect to the Lebesgue measure. However, to determine the irrationality exponent of a 
given real number is a very difficult problem, unless we know explicitly its continued fraction expansion. 
For example, the irrationality exponent of $\rme$
is equal to $2$. By Roth's Theorem, the irrationality exponent of any algebraic irrational real number is also equal to $2$. 
Classical numbers known to have their irrationality exponent equal to $2$ 
include non-zero rational powers of $\rme$, badly approximable numbers, 
$\tan \frac{1}{a}$, where $a$ is any positive integer, etc.
Further examples are given in \cite{Ad10,BuKim17}. 
Note that the irrationality exponents of classical numbers like 
$\pi$, $\zeta(2)$, $\zeta(3)$, $\log (2)$ remain unknown. At present, the best known estimate for $\pi$ is
$\mu (\pi) \le 7.10321$, established in \cite{ZeZu20}.

Throughout the present paper, $b$ always denotes an integer 
greater than or equal to $2$
and $\xi$ a real number. There exists a unique infinite sequence 
$(a_j)_{j \ge 1}$ of integers from $\{0, 1, \ldots, b-1\}$, called the
$b$-ary expansion of $\xi$, such that
\beq \label{xiexp}
\xi = \lfloor \xi \rfloor + \sum_{j\ge 1} \, \frac{a_j}{b^j}     
\eeq
and $a_j \not= b-1$ for infinitely many indices $j$. 
Here, $\lfloor \cdot \rfloor$ denotes the integer part function. 
Clearly, the sequence $(a_j)_{j \ge 1}$ is ultimately periodic if, and only if, $\xi$ is rational.

Let us introduce some terminology from combinatorics on words.
Let $\mathcal A$ be a finite set called an alphabet and denote by $|\mathcal A|$ its cardinality.
A word over $\mathcal A$ is a finite or infinite sequence of elements of $\mathcal A$.
For a (finite or infinite) word ${\bf x} = x_1 x_2 \ldots$ written over $\mathcal A$,
let $n \mapsto p (n, {\bf x})$ denote its subword complexity function
which counts the number of different subwords of length $n$ occurring in $\mathbf x$, that is,
$$
p (n,{\bf x}) = {\rm Card} \{ x_{j+1} x_{j+2} \dots x_{j+n} : j \ge 0 \}, \quad n \ge 1.
$$
Clearly, we have
$$
1 \le p(n, {\bf x}) \le |\mathcal A|^n, \quad n \ge 1.
$$
If ${\bf x}$ is ultimately periodic, then there exists an integer $C$ such that 
$p(n, {\bf x}) \le C$ for $n \ge 1$. Otherwise, we have
\beq \label{pincr}
p(n+1, {\bf x}) \ge p(n, {\bf x}) + 1, \quad n \ge 1, 
\eeq
thus $p(n, {\bf x}) \ge n+1$ for $n \ge 1$. There exist uncountably many infinite words 
${\bf s}$ over $\{0, 1\}$ such that $p(n, {\bf s}) = n+1$ for $n \ge 1$. These words are called
Sturmian words. Classical references on combinatorics on words and
on Sturmian sequences include \cite{Fogg02,Loth02,AlSh03}. 
The Fibonacci word $\bff$ defined in Section \ref{disc} is an emblematic example of a Sturmian word. 

A natural way to measure the
complexity of the real number $\xi$ written in base $b$ 
as in \eqref{xiexp} is to count the number of
distinct blocks of given length in the infinite word ${\bf a} = a_1 a_2 \ldots$. 
We set
$$
p(n, \xi, b) = p (n, \bfa), \quad n \ge 1. 
$$

A real number $\xi$ is normal to base $b$ if, for every $k \ge 1$, every block of $k$ digits over $\{0, 1, \ldots b-1\}$ occurs 
in the $b$-ary expansion of $\xi$ with the same frequency $1/b^k$. 
In particular, if $\xi$ is normal to base $b$, then $p(n, \xi, b) = b^n$ for every 
positive integer $n$. 
To establish a good lower bound for $p(n, \xi, b)$ is a first step towards 
the confirmation that $\xi$ is normal to base $b$.
This point of view has been taken by Ferenczi and Mauduit \cite{FeMa97}, who showed that 
the $b$-ary expansion of an irrational algebraic number cannot be too simple.

\begin{theorem}[Ferenczi and Mauduit, 1997]   \label{FM}
Let $\xi$ be a real number. 
If the sequence of $b$-ary digits of $\xi$ is a Sturmian sequence, then $\xi$ is transcendental.
\end{theorem}

It has been observed in \cite{Ad10} that the statements on the combinatorial structure of 
Sturmian words established in
\cite{BeHoZa06} and \cite{AdBu11} almost immediately imply that the irrationality exponent of any 
real number, whose sequence of digits in some integer base is Sturmian, 
is greater than $2$. In view of Roth's Theorem mentioned above, this extends 
Theorem \ref{FM}. 
The main result of \cite{BuKim17}, reproduced below, is much stronger and 
establishes a connection between the irrationality exponent of a 
real number and the complexity of its $b$-ary expansion. 
It asserts that any irrational real number, whose expansion in some integer 
base has sufficiently small complexity, has an exponent of irrationality larger than $2$.

\begin{theorem}[Bugeaud and Kim, 2017] \label{Pisa}
Let $b \ge 2$ be an integer and $\xi$ an irrational real number.
If $\mu$ denotes the irrationality exponent of $\xi$, then
$$
\liminf_{n \to + \infty} \, \frac{p(n, \xi, b)}{n} \ge 
1 +  \frac{1 - 2 \mu (\mu - 1) (\mu - 2)}{\mu^3 (\mu - 1)}.    
$$
and
$$
\limsup_{n \to + \infty} \, \frac{p(n, \xi, b)}{n}  
\ge  1  +  \frac{1 - 2 \mu (\mu - 1) (\mu - 2)}{3 \mu^3 - 6 \mu^2 + 4 \mu - 1}.  
$$
\end{theorem}

For an algebraic irrational real number $\xi$, Theorem \ref{Pisa} gives a much weaker 
result than the main theorem of \cite{AdBu07}, which asserts that, for any integer $b \ge 2$, we have
$$
\lim_{n \to + \infty} \, \frac{p(n, \xi, b)}{n} = + \infty ;
$$
see also \cite{BuEv08}.

The purpose of the present work is to show how Theorem \ref{Pisa} can be strengthened by 
means of a more precise combinatorial study than that given in \cite{BuKim17}. 
We establish the following result.

\begin{theorem} \label{PisaBis}
Let $b \ge 2$ be an integer and $\xi$ an irrational real number.
If $\mu$ denotes the irrationality exponent of $\xi$, then
\beq \label{newliminf}
\liminf_{n \to + \infty} \, \frac{p(n, \xi, b)}{n} \ge 
1+ \frac{-\mu^3 + 2 \mu^2 + \mu - 1}{\mu^4 - 2 \mu^3 + 3 \mu^2 - 3 \mu + 1}
\eeq
and
\beq \label{newlimsup}
\limsup_{n \to + \infty} \, \frac{p(n, \xi, b)}{n}  
\ge  \frac{\mu + \sqrt{ 4(\mu - 1)^3 + \mu^2 } }{2 \mu (\mu - 1)}.
\eeq
In particular, every irrational real number $\xi$ whose irrationality exponent
is equal to $2$ satisfies
$$
\liminf_{n \to + \infty} \, \frac{p(n, \xi, b)}{n}  \ge  \frac{8}{7} = 1.1428 \ldots 
\enspace \hbox{and} \enspace
\limsup_{n \to + \infty} \, \frac{p(n, \xi, b)}{n}  \ge \frac{1 + \sqrt{2}}{2} =   1.207 \ldots 
$$
\end{theorem}

Theorem \ref{PisaBis} 
improves \cite[Theorem 1.5]{BuKim17}, where, for $\mu = 2$,
we got the lower bounds
$9/8 = 1.125$ and $8/7$, respectively. 
Inequality \eqref{newliminf} gives a non-trivial result on the $b$-ary expansion of
any real number $\xi$ whose irrationality exponent satisfies
$$
2 \le \mu (\xi) <  \mu_1 := 2.246 \ldots, 
$$
where $\mu_1$ is the root greater than $2$ of the polynomial 
$X^3 - 2 X^2 - X + 1$.   
Inequality \eqref{newlimsup} gives a non-trivial result on the $b$-ary expansion of
any real number $\xi$ whose irrationality exponent satisfies
$$
2 \le \mu (\xi) <  \mu_2 := 2.324\ldots, 
$$
where $\mu_2$ is the root greater than $2$ of the polynomial $X(X-1)(X-2) - 1$.   
The ranges are larger than the one in \cite{BuKim17}, where $\mu(\xi)$ has to be smaller 
than $2.19 \ldots$ (which is the root greater than $2$ of the polynomial $2 X(X-1)(X-2) - 1$). 
Thus, Theorem \ref{PisaBis} applies to a slightly wider class
of classical numbers than Theorem \ref{Pisa}, which includes in particular 
the transcendental number $\log (1 + \frac{1}{a})$, where 
$a$ is a sufficiently large positive integer. More examples are given in \cite[Section 2]{BuKim17}.  
As noted in \cite{BuKim17}, the badly approximable number $\sum_{k \ge 0} 2^{-2^k}$ 
shows that Theorem \ref{PisaBis} is sharp up to the values of the numerical constants. 

A key ingredient for the proof of Theorems \ref{Pisa} and \ref{PisaBis} 
is the study of a complexity function 
defined in terms of the smallest return time of a factor of an infinite word. 
For an infinite word ${\bf x} = x_1 x_2 \dots $ and an integer $n \ge 1$, set 
$$ 
r(n,{\bf x}) = \min \{ m \ge 1 :   x_{i} \ldots x_{i+n-1} = x_{m-n+1} \ldots x_{m} 
\text{ for some } i \text{ in } \{1,  \ldots , m-n\} \} .
$$
Said differently, $r(n,\mathbf x)$ denotes the length of the smallest prefix of ${\bf x}$
containing two (possibly overlapping) occurrences of some word of length $n$. 

\begin{definition}
For an infinite word $\bfx$ which is not ultimately periodic, we set
$$
\rep(\bfx) = \liminf_{n \to + \infty} \, \frac{r(n, \bfx)}{n}  \quad \hbox{and} \quad
\Rep(\bfx) = \limsup_{n \to + \infty} \, \frac{r(n, \bfx)}{n}.
$$
\end{definition}

A key ingredient in the proof of Theorem \ref{Pisa}  
is an improvement of the 
trivial inequality $\Rep (\bfx) \ge \rep (\bfx)$. Namely, we established in \cite{BuKim17} that 
\beq    \label{RrepPisa}
\Rep (\bfx) \ge \rep (\bfx) + \frac{1}{1 + \rep (\bfx) + (\rep (\bfx))^2}.
\eeq
We improve this result in Section \ref{sec:3}, where we establish the following lower bound. 

\begin{proposition}  \label{RrepImpr} 
For an infinite word $\bfx$ which is not ultimately periodic, we have
$$
\Rep (\bfx) \ge \frac{\rep (\bfx) + 1}{2} + \frac{\sqrt{(\rep (\bfx))^2 (\rep (\bfx) - 1)^2  + 4(\rep (\bfx) - 1)}}{2 \, \rep (\bfx)}.
$$
\end{proposition}

Our paper is organized as follows. We gather several auxiliary combinatorial lemmas in Section \ref{sec:2}. 
Then, we establish Proposition \ref{RrepImpr} in Section \ref{sec:3} and complete the proof of 
Theorem \ref{PisaBis} in the subsequent section. We conclude the paper by some additional remarks.

\section{Auxiliary combinatorial lemmas} \label{sec:2}

The function $n \mapsto r(n, {\bf x})$ defined in Section \ref{sec:1} has been 
introduced and studied in \cite{BuKim19}, where the following two assertions 
are established. 

\begin{theorem}  \label{thsturm} 
For every infinite word ${\bf x}$ which is not ultimately periodic, there exist
arbitrarily large integers $n$ such that $r(n,{\bf x}) \ge 2n+1$. Consequently, we have $\Rep (\bfx) \ge 2$. The only 
infinite words ${\bf x}$ such that $r(n,{\bf x}) \le 2n+1$ for all $n \ge 1$
and which are not ultimately periodic are the Sturmian words. 
\end{theorem}

Lemma 3.1 of \cite{BuKim17}, reproduced below, shows how
the functions $n \mapsto p(n, \bfx)$ and $n \mapsto r(n, \bfx)$ are related.

\begin{lemma}\label{ubound}
For any infinite word ${\mathbf x}$ and any positive integer $n$, we have
$$
p(n,{\mathbf x}) \ge r(n,{\mathbf x}) - n. 
$$
\end{lemma}

We remark that any infinite word $\bfx$ such that $p(n,\bfx)/n$ tends to $1$ as $n$ tends to infinity 
satisfies $\Rep (\bfx) = 2$.

For a word  $U = u_1 \dots u_n$ composed of $n$ letters, 
set    
$$ 
\Lambda(U) = \{ 1 \le k < n :  u_i = u_{i + k} \text { for all }1 \le i \le n-k \}.    
$$
An element of $\Lambda(U)$ is called a period of $U$. 
A finite word $U$ is called {\it periodic}  if $\Lambda (U)$ is not empty. 
We stress that a period of a word of length $n$ may not be a divisor of $n$. 
The next lemma is a classical tool in combinatorics on words. 

\begin{lemma}[Fine and Wilf Theorem \cite{FW}]\label{fw}   
Let $U = u_1 \dots u_n$  
 and $h, k$ be in $\Lambda (U)$.  
If $n\ge h + k - \mathrm{gcd}(h,k)$,  
then $U$ is periodic of period $\mathrm{gcd}(h,k)$. 
\end{lemma}  

We conclude this section with an easy lemma, for which some notation is required. 

\begin{notation}
For positive integers $i, j$, we write 
$x_i^j$ for the factor $x_i x_{i+1} \ldots x_j$ of the word ${\bf x} = x_1 x_2 \ldots$ 
when $i \le j$ and, by convention, $x_i^j$ is the empty word when $i > j$. 
\end{notation}

\begin{lemma}\label{lll} 
Let $\bfx$ be an infinite word and $m, n$ be integers with $m > n \ge 1$. 
If $\lambda$ is in $\Lambda(\bfx_n^m)$, then $r(m-n+1-\lambda,{\mathbf x}) \le m$.
\end{lemma}

\begin{proof}
The assumption implies that $x_{n}^{m- \lambda} = x_{n + \lambda}^m$. Consequently, a same word of length 
$m-n+1-\lambda$ has two occurrences in $x_n^m$, thus in $x_1^m$.
\end{proof}

\section{Proof of Proposition \ref{RrepImpr}}  \label{sec:3}

Let $\mathbf x$ be an infinite word which is not ultimately periodic.  
For simplicity, we often write $r(\cdot)$ for $r(\cdot , \bfx)$. 
Set $\sigma = \Rep(\bfx)$ and $\rho = \rep(\bfx)$. 
Let $\eps$ be an arbitrarily small positive real number.
We have 
$$
\rho - \eps \le \dfrac{r(n, {\mathbf x})}{n} \le \sigma + \eps, \quad \hbox{for all sufficiently large $n$.}
$$
From Theorem~\ref{thsturm}, we note that $\Rep(\bfx) \ge 2$.
Therefore, we assume that $\rho \ge \rho_2$, 
where $\rho_2 := 1.754 \ldots$ is the real root of $X (X- 1)^2 - 1$.
This is equivalent to
$$
\frac{\rho+1}{2} + \frac{\sqrt{\rho^2(\rho-1)^2+4(\rho-1)}}{2\rho} \ge 2.
$$
Since 
$$
\rho + \frac{1}{1+\rho} > \frac{\rho+1}{2} + \frac{\sqrt{\rho^2(\rho-1)^2+4(\rho-1)}}{2\rho},
$$
we also assume that 
\begin{equation}\label{disj}
\sigma < \rho + \frac{1}{1 + \rho}.
\end{equation} 
Otherwise, Proposition~\ref{RrepImpr} holds immediately.
We will see in Section \ref{disc} that the inequality 
$\sigma \ge \rho + 1 / (1 + \rho)$ does not hold for all $\bfx$. 
This justifies the assumption \eqref{disj}.

Let $(n_j)_{j \ge 1}$ be the increasing sequence composed of all the 
jumps of the function $n \mapsto r( n, \bfx)$, that is, of all the integers $n$ such that $r(n + 1) > r(n) +1$. 
Between two consecutive jumps, the function $n \mapsto r( n, \bfx)$ has slope $1$. 
Consequently, for $j \ge 2$, we have 
$$
r(n_{j-1} + h) = r(n_{j-1} + 1) + h - 1, \quad h=1, \ldots , n_j - n_{j-1},
$$
thus, in particular, 
\beq \label{rcst}
r(n_{j-1} + 1) = r(n_{j}) - n_{j} + n_{j-1} + 1.
\eeq
Observe that 
$$
\rep(\bfx) = \liminf_{j \to + \infty} \, \frac{r(n_j, \bfx)}{n_j}, \quad
\Rep(\bfx) = \limsup_{j \to + \infty} \, \frac{r(n_j + 1, \bfx)}{n_j +1}.
$$
We further define
\beq \label{beta}
\beta (\bfx) =  \liminf_{j \to + \infty} \, \frac{r(n_j + 1, \bfx)}{n_j +1}.
\eeq

We will bound $\Rep(\bfx)$ and $\beta (\bfx)$ from below. 
To this end, we take a large integer $n$ in the sequence $(n_j)_{j \ge 1}$, that is, 
such that $r(n+1) > r(n) + 1$, and we bound $r(n+1) / (n+1)$ from below. 

Let $n$ be an integer such that $r(n+1) > r(n) +1$.
Let $\lambda, \lambda'$ denote the positive integers defined by 
$$
x_{r(n)-n+1}^{r(n)} = x_{r(n)-n+1-\lambda}^{r(n)-\lambda}, \quad 
x_{r(n+1)-n}^{r(n+1)} = x_{r(n+1)-n-\lambda'}^{r(n+1)-\lambda'}.
$$

It follows from \eqref{disj} that 
\begin{equation}\label{rndiff}
r(n+1) - r(n) \le \frac{n}{2}   
\end{equation}   
since
$$
r(n+1) - r(n) \le (\sigma - \rho + 2 \eps) n + \sigma + \eps < \left( \frac{1}{1 + \rho} +2\eps\right) n+ \sigma + \eps. 
$$

Observe that 
$$
\lambda = \min\Lambda \bigl(x_{r(n)-n+1-\lambda}^{r(n)} \bigr),  \quad
\lambda' = \min \Lambda \bigl(x_{r(n+1)-n-\lambda'}^{r(n+1)} \bigr).
$$
Set 
$$
V_n = x_{r(n+1)-n}^{r(n)} = x_{r(n+1)-n-\lambda}^{r(n)-\lambda} = x_{r(n+1)-n-\lambda'}^{r(n)-\lambda'}.
$$
The length $v_n$ of $V_n$ satisfies 
\beq \label{lgth}
v_n = r(n)-r(n+1)+n+1 < n. 
\eeq 
Note that the assumption \eqref{disj} guarantees that 
$$
v_n \ge \frac{\rho}{1 + \rho} n. 
$$
We have 
\begin{equation}\label{lambdadiff}
\lambda \ne \lambda',
\end{equation}
since otherwise we would get from \eqref{rndiff} that
$x_{r(n) - \lambda + 1} = x_{r(n) + 1}$, in contradiction to the 
definition of $r(n)$.  

Since there are two occurrences of $V_n$ in $x_1^{r(n) - \min\{\lambda, \lambda'\}}$, 
we have
\begin{align}
r(n)-\lambda \ge r(v_n), \ &\text{ if } \ \lambda < \lambda', \label{ineq1:case12} \\
r(n)-\lambda' \ge r(v_n), \ &\text{ if } \ \lambda > \lambda'. \label{ineq1:case34}
\end{align}
We distinguish five cases. 

\sm
\noindent
$\bullet$ 
Assume that $\lambda < \lambda'$.
If $\lambda \ge v_n$, then 
$$
r(v_n) \le r(n) - \lambda \le r(n) - v_n = r(n+1) - n - 1 \le (\sigma + \eps) (n+1) - n - 1
$$
and 
$$
r(v_n) \ge (\rho - \eps) v_n \ge (\rho - \eps) \bigl( (\rho - \eps) n - (\sigma + \eps) (n+1)  + n + 1 \bigr).     
$$
Consequently, we get
$$
(\rho - \eps)^2 n \le (\sigma + \eps) (\rho+1 - \eps) (n+1) - (\rho + 1 - \eps) (n+1), 
$$
a contradiction with \eqref{disj} if $\eps$ is small enough. Thus, we have $v_n > \lambda$.
Furthermore, 
$$
\hbox{$V_n = x_{r(n+1)-n-\lambda'}^{r(n)-\lambda'} = x_{r(n+1)-n}^{r(n)}$ is a subword of $x_{r(n)-n+1}^{r(n)}$},
$$
thus $\lambda$ is in $\Lambda \bigl(x_{r(n+1)-n-\lambda'}^{r(n)-\lambda'} \bigr)$. 
It then follows from Lemma~\ref{lll} that
\begin{equation}\label{ineq2:case12}
r(n) -\lambda' \ge r(v_n - \lambda).
\end{equation}

\sm
\noindent
$\bullet$ 
Assume that $\lambda > \lambda'$ and $r(n+1) -\lambda' < r(n) +1$.
If $\lambda' \ge n$, then 
$$
r(v_n) \le r(n) - \lambda' \le r(n) - n. 
$$
Arguing as above, we get as well a contradiction with \eqref{disj}.  
Thus, we have $n > \lambda'$. 
Furthermore,
$$
\hbox{$x_{r(n)-n+1-\lambda}^{r(n)-\lambda} = x_{r(n)-n+1}^{r(n)}$ is a subword of $x_{r(n+1)-n-\lambda'}^{r(n+1)}$, }
$$
thus $\lambda'$ is in $\Lambda \bigl( x_{r(n)-n+1-\lambda}^{r(n)-\lambda} \bigr)$. 
It then follows from Lemma~\ref{lll} that
\begin{equation}\label{ineq2:case3}
r(n) -\lambda \ge r(n - \lambda').
\end{equation}

\sm
\noindent
$\bullet$ Assume that $\lambda > \lambda'$ and $r(n+1) -\lambda' \ge r(n) +1$. Then,
$$ 
\hbox{$x_{r(n+1)-n-\lambda'-\lambda}^{r(n)-\lambda} = x_{r(n+1)-n-\lambda'}^{r(n)}$ is a subword of 
$x_{r(n+1)-n-\lambda'}^{r(n+1)}$,}
$$ 
thus $\lambda'$ is in 
$\Lambda \bigl( x_{r(n+1)-n-\lambda'-\lambda}^{r(n)-\lambda} \bigr)$.
It then follows from Lemma~\ref{lll}  that
\begin{equation}\label{ineq2:case4}
r(n) -\lambda \ge r(v_n).
\end{equation}

\sm
\noindent
$\bullet$ 
Assume that $r(n+1)-\lambda' \le r(n)-\lambda+1$. 
Then $x_{r(n)-n+1-\lambda}^{r(n)}$ is a subword of $x_{r(n+1)-n-\lambda'}^{r(n+1)}$,
thus $\lambda'$ is in $\Lambda \bigl( x_{r(n)-n+1-\lambda}^{r(n)} \bigr)$.
Since $\lambda$ is also in $\Lambda \bigl( x_{r(n)-n+1-\lambda}^{r(n)} \bigr)$, we deduce from
Lemma \ref{fw} that
\begin{equation}\label{ineq3:case1}
\lambda' > n+\lambda - \lambda +1 = n+1.
\end{equation}

\sm
\noindent
$\bullet$
Assume that $r(n+1)-\lambda' > r(n)-\lambda+1$. 
Then  $x_{r(n+1)-n-\lambda'}^{r(n)}$ is a subword of $x_{r(n)-n+1-\lambda}^{r(n)}$,  
thus $\lambda$ is in $\Lambda \bigl( x_{r(n)-n+1-\lambda'}^{r(n)} \bigr)$ and we see that
$$
\lambda, \lambda' \in \Lambda \bigl( x_{r(n+1)-n-\lambda'}^{r(n)} \bigr). 
$$
We deduce from
Lemma \ref{fw} that
$$ 
\lambda' > v_n+\lambda' -\lambda +1,
$$
that is, 
\begin{equation}\label{ineq3:case234}
\lambda > v_n+ 1.
\end{equation}

\medskip

These five cases can be summarized in the following four cases:

\sm
\noindent
$\bullet$ Case (i): $r(n+1)-\lambda' \le r(n)-\lambda+1$ (in this case we have $\lambda < \lambda'$). 

\sm

We have \eqref{ineq1:case12}, \eqref{ineq2:case12}, \eqref{ineq3:case1}, that is, 
$$
r(n)-\lambda \ge (\rho - \eps) v_n, \quad
r(n) -\lambda' \ge  (\rho - \eps) (v_n - \lambda), \quad
\lambda' > n+1.
$$ 
Consequently,
\begin{align*}
n+1  < \lambda' &\le r(n) -  (\rho - \eps) v_n +  (\rho - \eps) \lambda \\
& \le  r(n) -  (\rho - \eps) v_n +  (\rho - \eps) r(n) -  (\rho - \eps)^2 v_n,
\end{align*}
thus
$$
n+1 < (1+ \rho - \eps) r(n) - (\rho - \eps) (1+ \rho - \eps) v_n.
$$
Combining this with 
$$
(\rho - \eps) (1+ \rho - \eps) r(n+1) = (r(n) + n + 1 - v_n) (\rho - \eps) (1+ \rho - \eps), 
$$
given by \eqref{lgth}, we get 
$$
(\rho - \eps) (1+ \rho - \eps) r(n+1) >  
(n+1) (1 + (\rho - \eps) (1+ \rho - \eps) ) + r(n) (1 + \rho - \eps) (\rho - 1 - \eps), 
$$
thus,
\beq  \label{beta1}
\frac{r(n+1)}{n+1} > 1 + \frac{1}{ (\rho - \eps) (1+ \rho - \eps) } + \frac{r(n)}{n+1} \cdot \frac{\rho - 1 - \eps}{\rho - \eps}. 
\eeq
By letting $\eps$ tend to $0$, we obtain 
$$
\sigma -  \rho \ge \frac{1}{\rho (\rho +1)}.
$$

\sm
\noindent 
$\bullet$ Case (ii): $r(n)-\lambda +1< r(n+1)-\lambda'$ and $\lambda < \lambda'$. 

\sm

We have \eqref{ineq1:case12}, 
 \eqref{ineq3:case234}.
We get
$$
r(n)-\lambda \ge (\rho - \eps)  v_n, \quad 
\lambda > v_n+ 1. 
$$ 
Consequently,
$$
(1 + \rho - \eps)  v_n < r(n) - 1,
$$
thus, by \eqref{lgth}, we obtain
$$
(1 + \rho - \eps) r(n+1) > (1+\rho - \eps)  (r(n) + n + 1) - r(n) + 1
$$
and
\beq  \label{beta2}
\frac{r(n+1)}{n+1} > 1 +  \frac{r(n)}{n+1} \cdot \frac{\rho  - \eps}{1 + \rho - \eps}. 
\eeq
This gives eventually
$$
\sigma - \rho \ge  \frac{1}{\rho +1}.
$$

\sm
\noindent 
$\bullet$ Case (iii): $r(n+1)-\lambda' <  r(n) + 1$ and $\lambda > \lambda'$. 

\sm

We have \eqref{ineq1:case34}, \eqref{ineq2:case3}, \eqref{ineq3:case234}, that is, 
$$
r(n)-\lambda' \ge (\rho - \eps) v_n,  \quad
r(n) -\lambda \ge (\rho - \eps) (n - \lambda'),  \quad
\lambda > v_n+ 1. 
$$ 
Consequently,
$$
v_n + 1 < \lambda \le r(n) - (\rho - \eps) n + (\rho - \eps) \lambda' \le r(n) - (\rho - \eps) n +
(\rho - \eps) r(n) - (\rho - \eps)^2 v_n,
$$
thus
$$
(1+ (\rho - \eps)^2) v_n + (\rho - \eps) n + 1 < (1 + \rho - \eps) r(n).
$$
By using by \eqref{lgth}, this implies 
$$
(1+ (\rho - \eps)^2)  r(n+1) > (1+ (\rho - \eps)^2) (r(n) + n +1) + (\rho - \eps) n  - (1 + \rho - \eps) r(n),
$$
thus 
\beq  \label{beta3}
\frac{r(n+1)}{n+1} > 1 +  \frac{n}{n+1} \cdot \frac{\rho - \eps}{1+ (\rho - \eps)^2} +
\frac{r(n)}{n+1} \cdot \frac{(\rho - \eps)^2 - \rho  + \eps}{1+ (\rho - \eps)^2}. 
\eeq
This gives eventually
$$
\sigma - \rho \ge \frac{1}{\rho^2 +1}.
$$

\sm
\noindent 
$\bullet$ Case (iv): $r(n+1)-\lambda' \ge r(n)+1$ and $\lambda > \lambda'$. 

\sm

We have \eqref{ineq2:case4}, \eqref{ineq3:case234}, that is,
$$
r(n) -\lambda \ge (\rho - \eps) v_n,  \quad
\lambda > v_n+ 1. 
$$ 
As in Case (ii), we get
$$
\sigma - \rho \ge \frac{1}{\rho +1}.
$$

\medskip

To summarize, we have established that 
$$
\sigma - \rho
\ge \begin{cases}
\dfrac{1}{\rho(\rho+1)},  &\text{ in Case (i)},\\
\dfrac{1}{\rho^2+1},   &\text{ in Case (iii)},\\
\dfrac{1}{\rho+1},  &\text{ in Cases (ii) and (iv)}.\\
\end{cases}
$$
This improves \eqref{RrepPisa}, 
but we will do slightly better below. 
It may be tempting to conjecture that Case (ii) or Case (iv) occurs infinitely often,
but the discussion in Section \ref{disc} shows that this is not true in whole generality.

But before going further, let us observe that, by \eqref{beta1}, \eqref{beta2}, and \eqref{beta3}, we get 
\beq   \label{liminfbeta} 
\beta(\bfx) \ge \rho + \dfrac{1}{\rho(\rho+1)}. 
\eeq

For $j \ge 1$, let $\lambda_j$ be the positive integer such that 
$$
x^{r(n_j)}_{r(n_j)-n_j+1} = x^{r(n_j) - \lambda_j}_{r(n_j)-n_j+1 - \lambda_j}.
$$
By \eqref{lambdadiff}, there exist infinitely many integers $j$ such that $\lambda_{j+1} > \lambda_j$. 
Hence, there are infinitely many integers $n$ for which we are in Case (i) or in Case (ii).

Therefore, Case (i) happens infinitely often 
and we consider now the subsequent jump. 
This means that we take
integers $n$ and $n+k$ such that 
$$
r(n+1) > r(n) +1, \ r(n+k+1) > r(n+k) +1, \ 
r(n+k) = r(n+1) + k-1.
$$
This implies that 
$r(m+1) = r(m) + 1$ for $m=n+1, \ldots , n+k-1$. 
Let $\lambda, \lambda', \lambda''$ be the integers satisfying
$$
x_{r(n)-n+1}^{r(n)} = x_{r(n)-n+1-\lambda}^{r(n)-\lambda},  \quad x_{r(n+1)-n}^{r(n+1)} 
= x_{r(n+1)-n-\lambda'}^{r(n+1)-\lambda'}, 
$$
$$
x_{r(n+k+1)-n-k}^{r(n+k+1)} = x_{r(n+k+1)-n-k-\lambda''}^{r(n+k+1)-\lambda''}. 
$$

Since $r(n+1)+k-1 = r(n+k) \ge (\rho - \eps) (n+k)$, we get
\begin{equation}\label{kbound}
k \le \frac{ r(n+1) - (\rho - \eps) n -1}{\rho - \eps -1}.
\end{equation}
By letting $\eps$ tend to $0$, this gives
\begin{equation}\label{kboundbis}
k \le \frac{ \sigma  - \rho}{\rho -1} \, n.
\end{equation}

Since we are in Case (i), we get \eqref{ineq3:case1}, that is,
\beq \label{minlambda}
\lambda' > n + 1.
\eeq
For the subsequent jump, we are in one of the cases 
(i), (ii), (iii), or (iv). Since Cases (ii) and (iv) can occur 
only finitely often, we have only to consider 
Cases (i) and (iii). 

\sm

\noindent 
$\bullet$ Case (i) for the subsequent jump: $r(n+k+1)-\lambda'' \le r(n+k)-\lambda'+1$ (in this case we have $\lambda' < \lambda''$).

\sm

Then, by \eqref{ineq1:case12} we have
$$
r(n+k)-\lambda' \ge (\rho - \eps) v_{n+k} = (\rho - \eps)  (r(n+k) - r(n+k+1) +n+k+1),
$$
thus, by \eqref{minlambda}, 
\begin{align*}
(\rho - \eps)  r(n+k+1) &\ge (\rho - \eps -1) r(n+k) + (\rho - \eps)  (n+k+1) +\lambda' \\
&> (\rho - \eps)  (\rho - \eps -1) (n+k) + (\rho - \eps)  (n+k+1) + n+1 \\
& = (\rho - \eps)^2 (n+k) + \rho  - \eps + n+1.
\end{align*}
Combined with \eqref{kboundbis}, this gives
\begin{align*}
\frac{r(n+k+1)}{n+k+1} 
& > \frac{(\rho - \eps)(n+k) + 1}{n+k+1} + \frac{n+1}{(\rho - \eps) (n+k+1)}  \\
& \ge 
\frac{(\rho - \eps)(n+k) + 1}{n+k+1} + \frac{n+1}{(\rho-\eps) (n + \frac{ \sigma - \rho}{\rho -1}n +1)}.
\end{align*}
By letting $\eps$ tend to $0$, we get 
$$ 
\sigma
\ge \rho + \frac{\rho-1}{\rho(\sigma -1)}.
$$

\sm
\noindent 
$\bullet$ Case (iii) for the subsequent jump:  $r(n+k+1)-\lambda'' <  r(n+k) + 1$ and $\lambda' > \lambda''$.

\sm

Then, by \eqref{ineq1:case34} and \eqref{ineq2:case3} 
we have
$$
r(n+k)- (\rho - \eps) v_{n+k} \ge \lambda'' \ge n+k + \frac{\lambda'}{\rho - \eps} - \frac{r(n+k)}{\rho - \eps}. 
$$
Instead of the lower bound $\lambda' > v_{n+k} + 1$, we use \eqref{minlambda} and obtain
\begin{multline*}
(\rho - \eps) r(n+k)-  (\rho - \eps)^2  (r(n+k) - r(n+k+1) +n+k+1) \\
 > (\rho - \eps) (n+k)  + n + 1 - r(n+k). 
\end{multline*}   
Letting $\eps$ tend to $0$, this gives
$$
\rho^2 r(n+k+1) 
\ge (\rho^2 - \rho - 1) r(n+k)  +  \rho (\rho +1) (n+k) + n+1 + \rho^2.
$$
By dividing both members of the inequality by $n+k+1$, using \eqref{kboundbis}, we eventually arrive at 
$$ 
\sigma \ge \rho + \frac{\rho-1}{\rho^2( \sigma -1)}.
$$

\bigskip

Consequently, we have shown that, in every case, we have
$$
\sigma \ge \rho + \frac{\rho - 1}{\rho^2 (\sigma - 1)}.
$$
This proves Proposition \ref{RrepImpr}.
This can be rewritten as
\beq \label{mino} 
\sigma \ge \frac{\rho + 1}{2} + \frac{\sqrt{\rho^2 (\rho - 1)^2  + 4(\rho - 1)}}{2 \rho}.
\eeq
A rapid calculation shows that the right hand side of \eqref{mino} is larger than $\rho + 1 / (\rho^2 + 1)$ for $\rho \ge \rho_2$.

\section{Proof of Theorem \ref{PisaBis}}\label{sec:4}

Let $\eps >0$ be a given (small) real number. 
Set $\sigma = \Rep(\bfx)$, $\rho = \rep(\bfx)$ as before. 
Then we have 
$$
\rho - \eps \le \dfrac{r(n, {\mathbf x})}{n} \le \sigma + \eps, \quad \hbox{for all sufficiently large $n$.}
$$
Observe that, by definition of the sequence $(n_j)_{j \ge 1}$ composed of all the 
jumps of the function $n \mapsto r( n, \bfx)$, we have  \eqref{rcst}, that is, 
$$
r(n_{j+1}, {\bf x}) = r(n_j + 1, {\bf x}) + n_{j+1} - n_j - 1 \ge (\rho - \eps) n_{j+1}.
$$
Consequently,
\begin{equation}\label{nj1bound}
n_{j+1} \le \frac{r(n_j + 1, {\bf x}) - n_j - 1}{\rho - 1 - \eps}.
\end{equation}
Let $n$ be an integer with $n_j + 1 \le n \le n_{j+1}$. By \eqref{pincr}
and Lemma \ref{ubound} we have
$$
p(n, {\bf x}) \ge p(n_j + 1, {\bf x}) + n - n_j - 1 \ge r(n_j + 1, {\bf x}) + n - 2 n_j - 2,
$$
thus
$$
\frac{p(n, {\bf x})}{n} \ge 1 + \frac{r(n_j + 1, {\bf x})   - 2 n_j - 2}{n}
\ge 1 + \frac{r(n_j + 1, {\bf x})   - 2 n_j - 2}{n_{j+1}}. 
$$
Combined with \eqref{nj1bound}, this gives
\begin{align*}
\frac{p(n, {\bf x})}{n} & \ge 1 + (\rho - 1 - \eps) \,  
\frac{r(n_j + 1, {\bf x})   - 2 n_j - 2}{r(n_j + 1, {\bf x})   - n_j - 1} \\
& \ge \rho - \eps - (\rho - 1 - \eps) \,  \frac{1}{\frac{r(n_j + 1, {\bf x})}{n_j + 1} - 1} \\
& \ge \rho - \eps - (\rho - 1 - \eps) \,  \frac{1}{ \beta(\bfx) - \eps - 1},   
\end{align*}
for all sufficiently large $j$, where $\beta(\bfx)$ is defined in \eqref{beta}. 
Since $\eps$ can be taken arbitrarily small, we deduce from \eqref{liminfbeta} that 
\beq \label{liminfpn}
\liminf_{n \to + \infty} \, \frac{p(n, {\bf x})}{n} \ge \rho - \frac{\rho - 1}{\rho + \dfrac{1}{\rho(\rho+1)}  - 1} 
= \rho \cdot \frac{ \rho^3 - \rho^2 - \rho + 2 }{\rho^3 - \rho +1}. 
\eeq
Note that \eqref{liminfbeta} and \eqref{liminfpn} give trivial bounds if $\rho < \rho_1$, where $\rho_1 : = 1.8019\dots $ is the root greater than 1 of the polynomial $X^3-X^2-2X+1$.

\medskip

Let $b \ge 2$ be an integer. 
Our last auxiliary result establishes a close connection between the exponent of repetition 
of an infinite word ${\mathbf x}$ written over $\{0, 1, \ldots , b-1\}$ and the
irrationality exponent (see Definition \ref{defirrexp}) 
of the real number whose $b$-ary expansion is given by ${\bf x}$. 
This is \cite[Lemma 3.6]{BuKim17}. 

\begin{lemma}\label{minmurep}
Let $b \ge 2$ be an integer and ${\bf x} = x_1 x_2 \ldots$ an infinite word over
$\{0, 1, \ldots , b-1\}$, which is not ultimately periodic.   
Then, the irrationality exponent of the irrational number 
$\sum_{k \ge 1} \, \frac{x_k}{b^k}$ satisfies 
\beq \label{murep}
\mu \Bigl( \sum_{k \ge 1} \, \frac{x_k}{b^k} \Bigr) \ge \frac{\rep({\mathbf x})}{\rep({\mathbf x}) - 1},
\eeq
where the right hand side is infinite if $\rep({\mathbf x}) = 1$. 
\end{lemma}

Lemma \ref{minmurep} shows that, when the exponent of repetition
of an infinite word ${\bf x} = x_1 x_2 \ldots$ is less than $2$, then the irrationality exponent of the associated real
number $\xi := \sum_{k \ge 1} \, {x_k}/{b^k}$ exceeds~$2$. 
Indeed, there are $\eps > 0$ and infinitely many pairs of positive integers $(u, v)$
such that 
$$
\| b^u (b^v - 1) \xi \| < \frac{1}{(b^u (b^v - 1))^{1 + \eps}} . 
$$
This does not mean, however, that all the 
(or all but finitely many) best rational approximations to $\xi$ can be read off its 
$b$-ary expansion. In particular, we do not know if (nor under which additional assumptions) 
equality holds in \eqref{murep}. 

We are in position to complete the proof of 
Theorem~\ref{PisaBis}.

\goodbreak

\noindent {\it Proof of Theorem~\ref{PisaBis}.}

Let $b \ge 2$ be an integer and $\xi$ an irrational real number. 
Put $\mu = \mu (\xi)$. 
Write $\xi$ in base $b$ as in \eqref{xiexp} and put 
${\mathbf a} = a_1 a_2 \ldots$. 
Lemma \ref{minmurep} asserts that
$$
\rep({\mathbf a}) \ge \frac{\mu}{\mu - 1}.
$$
Combined with \eqref{mino}, this gives
$$
\Rep({\mathbf a}) \ge \frac{\mu (2 \mu - 1) + \sqrt{ 4(\mu - 1)^3 + \mu^2}}{2 \mu (\mu - 1)}.
$$
By Theorem \ref{thsturm}, this is non-trivial as soon as the lower bound is greater than or equal to $2$, that is, as soon as $\mu$ 
is less than the root $\mu_2$.
Observe that $\mu_i = \rho_i / (\rho_i - 1)$ for $i =1,2$.
This gives
$$
\limsup_{n \to + \infty} \, \frac{p(n, {\mathbf a})}{n} \ge \Rep({\mathbf a}) - 1
\ge \frac{\mu + \sqrt{ 4(\mu - 1)^3 + \mu^2 } }{2 \mu (\mu - 1)}.
$$
As well, by \eqref{liminfpn}, we obtain
$$
\liminf_{n \to + \infty}  \, \frac{p(n, {\mathbf a})}{n} 
\ge  \rho \cdot \frac{ \rho^3 - \rho^2 - \rho + 2 }{\rho^3 - \rho +1} 
\ge  \frac{\mu}{\mu - 1} \cdot \frac{\mu^3 - 3 \mu^2 + 5 \mu - 2}{\mu^3 - \mu^2 + 2 \mu - 1}. 
$$
We have 
completed the proof of 
Theorem~\ref{PisaBis}.

\section{Final discussion}  \label{disc}

The Fibonacci word
$$
\bff = f_1 f_2 f_3 \ldots = 010010100100101001010 \ldots 
$$
is defined as the limit of the sequence of finite words 
$(\bff_j)_{j \ge 1}$, where $\bff_1 = 0$, $\bff_2 = 01$, and $\bff_{j+2} = \bff_{j+1} \bff_j$, for 
$j \ge 1$. Clearly, the length of $\bff_j$ is equal to the Fibonacci number $F_j$ for $j \ge 1$. 
We check that, for $n \ge 3$, we have
$$
r(F_{n}-2, \bff) = F_{n+1} - 2, \quad r(F_{n}-1, \bff) = 2 F_{n} - 1, 
$$
and 
$$
r(F_n + h, \bff) = 2 F_n  + h, \quad h=-1, 0, \ldots , F_{n-1} - 2.
$$
We derive that $\rep(\bff) =  (1 + \sqrt{5})/2$ and $\Rep (\bff) = 2$, thus 
$$
\Rep (\bff) = \rep(\bff) + \frac{1}{1 + \rep(\bff)}.
$$
As noted in the course of Section \ref{sec:3}, it could be tempting to conjecture that 
$$
\Rep (\bfx) - \rep(\bfx) \ge  \frac{1}{1 + \rep(\bfx)},
$$ 
for every $\bfx$ which is not ultimately 
periodic. This is however not true. 
Indeed, we proved in \cite{BuKim19} the existence 
of a Sturmian word $\bfs$ such that $\rep(\bfs) = \sqrt{10} - \frac{3}{2}$. 
A rapid calculation shows that 
$$
\rep(\bfs) + \frac{1}{1 + \rep(\bfs)} 
= \frac{43}{39} \sqrt{10} - \frac{113}{78} = 2.037\ldots ,
$$ 
while Theorem \ref{thsturm} asserts that $\Rep (\bfs) = 2$. 

The difficulty for estimating the gap between $\Rep (\bfx)$ and $\rep(\bfx)$ comes from the following fact. 
At a jump $n_j$ of the function $n \mapsto r(n, \bfx)$, the second occurrence of a word of length $n_{j} +1$ 
in $\bfx_1^{r(n_{j} +1, \bfx)}$ may
overlap the second occurrence of a word of length $n_{j}$ in $\bfx_1^{r(n_{j}, \bfx)}$.  
If there are no such overlaps when $n_j$ is sufficiently large, 
then we say that the word $\bfx$ has the disjointness property and we have 
$$
r(n_{j} + 1, \bfx) - n_{j} - 1 \ge  r(n_{j}, \bfx),
$$
hence, 
\beq     \label{Rep1rep} 
\Rep (\bfx) \ge  \rep (\bfx) + 1.
\eeq
This disjointness property is automatically satisfied if, instead of looking for repetitions of an arbitrary word, 
we consider only repetitions of 
the digit $0$, that is, if we look only at large blocks of $0$. In that case, $r(n, \bfx)$ 
is replaced by the length of the shortest prefix 
of $\bfx$ containing two occurrences of $0^n$. This special case corresponds to the 
approximation by rational numbers whose denominator is a power of $b$ and has 
been studied in \cite{BuLi16}. The inequality \eqref{Rep1rep} then corresponds to the 
inequality $v_b \ge \hv_b / (1 - \hv_b)$ proved in \cite{BuLi16}. Here, as noted 
below Lemma \ref{minmurep}, we consider 
approximation by rational numbers whose denominator is of the form $b^u (b^v - 1)$, for 
positive integers $u$ and $v$. 
This explains why the combinatorial analysis is much more 
delicate in the present case than in \cite{BuLi16}. 

\section*{Acknowledgements}

This paper was written during the second author’s visit to IRMA, supported by the CNRS.
The second author was also supported by the National Research Foundation of Korea (RS-2023-00245719).

\

\end{document}